\documentclass[11pt]{amsart}

\usepackage{comment, amsmath, amssymb, amsgen, amsthm, amscd, xspace, color, epsfig}
\usepackage[hypertex]{hyperref}

\makeatletter
\newcommand{\rmnum}[1]{\romannumeral #1}
\newcommand{\Rmnum}[1]{\expandafter\@slowromancap\romannumeral #1@}

\newcommand{\del}{\delta}

\newcommand{\pt}{\partial}
\newcommand{\vf}{\varphi}

\newcommand{\Hom}{\operatorname{Hom}}

\newcommand{\ra}{\rightarrow}

\newcommand{\pf}{\begin{proof}}
\newcommand{\epf}{\end{proof}}
\newcommand{\eq}{\begin{equation}}
\newcommand{\eeq}{\end{equation}}
\newcommand{\eqn}{\begin{equation*}}
\newcommand{\eeqn}{\end{equation*}}

\newcommand{\frb}{\mathfrak{b}}

\newcommand{\frg}{\mathfrak{g}}
\newcommand{\frh}{\mathfrak{h}}

\newcommand{\frn}{\mathfrak{n}}

\newcommand{\frsl}{\mathfrak{sl}}

\newcommand{\bbC}{\mathbb{C}}

\newcommand{\bbN}{\mathbb{N}}

\newcommand{\bbZ}{\mathbb{Z}}
\newcommand{\caA}{\mathcal{A}}

\newcommand{\caO}{\mathcal{O}}

\newtheorem{theorem}[equation]{Theorem}%[section]
\newtheorem{cor}[equation]{Corollary}
\newtheorem{prop}[equation]{Proposition}
\newtheorem{lemma}[equation]{Lemma}

\theoremstyle{remark}
\newtheorem{remark}[equation]{Remark}

\theoremstyle{definition}
\newtheorem{definition}[equation]{Definition}

\newtheorem{example}[equation]{Example}

\numberwithin{equation}{section} 
\begin{document}

\title[Differential equations and singular vectors in Verma modules]{Differential equations and singular vectors in Verma modules over $\frsl(n,\bbC)$}

\author{Wei Xiao}
\thanks{This work is supported by NSFC Grant No. 11326059.}
\address{College of Mathematics and Computational Science, Shenzhen University,
Shenzhen, 518060, Guangdong, China}
\email{xiaow@szu.edu.cn}



\bigskip

\begin{abstract}
Xu introduced a system of partial differential equations to investigate singular vectors in the Verma modules of highest weight $\lambda$ over $\mathfrak{sl}(n,\mathbb{C})$. He proved that the solution space of this system in the space of truncated power series is spanned by $\{\sigma(1)\ |\ \sigma\in S_n\}$. We present an explicit formula of the solution $s_\alpha(1)$ for every positive root $\alpha$ and showed directly that $s_\alpha(1)$ is a polynomial if and only if $\langle\lambda+\rho,\alpha\rangle$ is a nonnegative integer. From this, we can recover a formula of singular vectors given by Malikov et al.
\end{abstract}
%\subjclass[2000]{Primary 22E46}

\maketitle

%
%%%%%%%%%%%%%%%%%%%%%%%%%%%%%%%%%%%%%%%%%%%%%%%%%%%%%%%%%%%%%%%%%%%%
%
\section{Introduction}
%
%%%%%%%%%%%%%%%%%%%%%%%%%%%%%%%%%%%%%%%%%%%%%%%%%%%%%%%%%%%%%%%%%%%%
%

This work concerns itself with Verma modules and corresponding partial differential equations. The study of Verma modules was initiated by Verma \cite{V} who showed that any nonzero homomorphism between Verma modules is injective and occurs with multiplicity one. He also found a sufficient condition for the existence of nontrivial homomorphism between Verma modules and conjectured that this condition is also necessary. The conjecture was ultimately proved by Bernstein-Gelfand-Gelfand \cite{BGG1} who introduced the well-known category $\caO$ to study representations of complex semismiple Lie algebras \cite{BGG2}.

One remaining problem in this direction is how to explicitly construct such a homomorphism if it exists. In fact, the homomorphism is completely determined by a weight vector called singular vector, which is contained in Verma module and can be
annihilated by positive root vectors. A general construction of singular vectors in Verma modules was emerged in \cite{MFF}. With this, they obtained an explicit formula of singular vectors for type $A_n$ in a PBW basis. The case of quantum group were considered in \cite{D, DF1, DF2}. In \cite{D}, singular vectors were given in a different basis of the universal enveloping algebra for $A_n$ and many cases of other types, with respect to so-called straight roots. In \cite{DF1, DF2}, the singular vectors for types $A_n$ and $D_n$ in a PBW basis were also obtained by their approach.

In \cite{Xu}, Xu built an identification between Verma modules and a space of polynomials, and the action of $\frsl(n,\bbC)$ on the Verma modules turns out to be a differential operator action on the polynomial space. Through this identification, a singular vector in the Verma modules corresponds to a polynomial solution of a system of second-order linear partial differential equations. In order to solve this system, he extended the action of $\frsl(n,\bbC)$ to a larger space of certain truncated formal power series, on which the complex power of negative simple root vectors are well-defined. Then he gave a differential-operator representation of symmetric group $S_n$ on the space of truncated formal power series. At last, he showed that the solution space of the system in the space of truncated formal power series is spanned by $\{\sigma(1)\ |\ \sigma\in S_n\}$. In particular, those $\sigma(1)$ that are polynomials determine the singular vectors in Verma modules over $\frsl(n,\bbC)$.

In the present paper, we derive an explicit formula of the solution $s_\alpha(1)$ (Theorem \ref{main thm1}) for any positive root $\alpha$ from Xu's results. With this formula in hand, the condition that $\sigma(1)$ is a polynomial if and only if $\langle\lambda+\rho,\alpha\rangle$ is a nonnegative integer can be verified directly. Here $\rho$ is half the sum of positive roots of $\frsl(n,\bbC)$. Moreover, we can obtain an explicit formula of singular vector corresponding to the condition that $\langle\lambda+\rho,\alpha\rangle$ is a nonnegative integer (Theorem \ref{main2}). This is in essence the result given in \cite{MFF}. Thanks to the theorems of Verma and BGG, it is enough for us to give all the singular vectors in Verma modules over $\frsl(n,\bbC)$.

The paper is organized as follows. In Sect. 2, we recall the notions and properties of Verma modules. In Sect. 3, we recall the main result in \cite{Xu} about the differential system and singular vectors in Verma modules over $\frsl(n,\bbC)$. In Sects. 4, we give the formula of $s_\alpha(1)$ and showed directly that it is a polynomial if and only if $\langle\lambda+\rho,\alpha\rangle$ is a nonnegative integer. An explicit formula of singular vector in this case is also given in this section.

%
%%%%%%%%%%%%%%%%%%%%%%%%%%%%%%%%%%%%%%%%%%%%%%%%%%%%%%%%%%%%%%%%%%%%
%
\section{Verma modules}
%
%%%%%%%%%%%%%%%%%%%%%%%%%%%%%%%%%%%%%%%%%%%%%%%%%%%%%%%%%%%%%%%%%%%%
%

In this section we recall the definition and some basic facts about Verma modules. Start with a complex semisimple Lie algebra $\frg$ and a fixed Cartan subalgebra $\frh$ of $\frg$. Let $\Phi\subseteq\frh^*$ be the root system of $\frg$ relative to $\frh$. Let $\frg_\alpha$ be the root subspace for the root $\alpha\in\Phi$. We choose a positive system $\Phi^+\subset\Phi$ with a corresponding simple system $\Delta\subseteq\Phi^+$. Then we have a Cartan decomposition $\frg=\bar\frn\oplus\frh\oplus\frn$ with $\frn=\oplus_{\alpha\in\Phi^+}\frg_\alpha$ and
$\bar\frn=\oplus_{\alpha\in\Phi^+}\frg_{-\alpha}$. Denote by $\frb=\frh\oplus\frn$ the corresponding Borel subalgebra of $\frg$.

Let $v_\lambda$ be a 1-dimensional $\frb$-module of weight $\lambda\in\frh^*$. The {\it Verma module} of highest weight
$\lambda$ is
\begin{equation*}
M(\lambda):=U(\frg)\otimes_{U(\frb)}v_\lambda.
\end{equation*}

Let $\langle,\rangle$ be the usual bilinear form on $\frh^*$ and
$\alpha^\vee=2\alpha/\langle\alpha,\alpha\rangle$. Denote by $W$ the Weyl group associated with the root
system $\Phi$. The {\it dot action} of
$W$ on $\frh^*$ is defined by $w\cdot\lambda=w(\lambda+\rho)-\rho$ for
$\lambda\in\frh^*$, where $\rho=\frac{1}{2}\sum_{\alpha\in\Phi^+}{\alpha}$.

Given $\lambda,\mu\in\frh^*$, denote $\mu\uparrow\lambda$ if there exists a positive root $\beta$ so that $\mu=s_\beta\cdot\lambda$ and
$\langle\lambda+\rho,\beta^\vee\rangle\in\mathbb{Z}^{>0}$. More generally, we say that $\mu$ is {\it strongly linked} to $\lambda$ and
write $\mu\uparrow\lambda$ if $\mu=\lambda$ or there exist $\beta_1,
\ldots, \beta_r\in\Phi^+$ such that
\[
\mu=(s_{\beta_1}\ldots
s_{\beta_r})\cdot\lambda\uparrow(s_{\beta_2}\ldots
s_{\beta_r})\cdot\lambda\uparrow\ldots\uparrow
s_{\beta_r}\cdot\lambda\uparrow\lambda.
\]

The following well-known results are due to Verma
\cite{V} and BGG \cite{BGG1} (see also Humphreys \cite{H}).

\begin{theorem}\label{BGG theorem}
Let $\lambda,\mu\in\frh^*$.
\begin{itemize}
\item [$\mathrm{(\rmnum{1})}$] Any nonzero homomorphism $\vf:M(\mu)\ra M(\lambda)$ is injective.
\item [$\mathrm{(\rmnum{2})}$] In all cases, $\dim\Hom_\frg(M(\mu),M(\lambda))\leq1$.
\item [$\mathrm{(\rmnum{3})}$] The hom space $\Hom_\frg(M(\mu),M(\lambda))\neq0$ if and only if $\mu$
is strongly linked to $\lambda$.
\end{itemize}
\end{theorem}
%
%%%%%%%%%%%%%%%%%%%%%%%%%%%%%%%%%%%%%%%%%%%%%%%%%%%%%%%%%%%%%%%%%%%%
%
\section{Differential equations and Verma modules}
%
%%%%%%%%%%%%%%%%%%%%%%%%%%%%%%%%%%%%%%%%%%%%%%%%%%%%%%%%%%%%%%%%%%%%
%
In this section, we outline the main results in \cite{Xu} used in this paper.

\begin{definition}
We say that a weight vector $v\in M(\lambda)$ is a {\it singular vector} if $\frn\cdot v=0.$
\end{definition}

\begin{example}\label{highest weight vector}
Let $\alpha$ be a positive simple root (that is, $\alpha\in\Delta$) and $E_{-\alpha}$ be a nonzero root vector in $\frg_{-\alpha}$. Suppose that
$n:=\langle\lambda+\rho,\alpha^\vee\rangle\in\bbZ^{>0}$.
Then $E_{-\alpha}^nv_\lambda$ is a singular vector in $M(\lambda)$.
\end{example}

From now on we let $\frg=\frsl(n,\bbC)$. Denote by $E_{i,j}$ the $n\times n$ matrix with $1$ in the $(i,j)$ position and $0$ elsewhere. Then the elements $$H_i=E_{i,i}-E_{i+1,i+1},\quad (i=1,2,\ldots,n-1)$$ form a basis of $\frh$. Let
\[
\{E_{i,j}\ |\ 1\leq i<j\leq n\}\ \mbox{and}\ \{E_{i,j}\ |\ 1\leq j<i\leq n\}
\]
be the sets of positive root vectors and negative root vectors respectively. If we denote by $e_i$ the function on $\sum_{j=1}^n\bbC E_{j,j}$ such that $e_i(E_{j,j})=\del_{ij}$, then the corresponding positive roots and negative roots should be
\[
\{e_i-e_j\ |\ 1\leq i<j\leq n\}\ \mbox{and}\ \{e_i-e_j\ |\ 1\leq j<i\leq n\}.
\]
The set of positive simple root vectors is
\[
\{E_{i,i+1}\ |\ i=1,2,\ldots,n-1\}
\]
with corresponding positive simple roots
\[
\{e_i-e_{i+1}\ |\ i=1,2,\ldots,n-1\}.
\]

Following the PBW theorem, we have $U(\frg)\simeq U(\bar\frn)\otimes
U(\frb)$. Then one can write $M(\lambda)\simeq
U(\bar\frn)v_\lambda$ as a left $U(\bar\frn)$-module. Let $\bbN$ be the additive semigroup of nonnegative integers. Denote by $$\Gamma:=\sum_{1\leq j<i\leq n}\bbN \epsilon_{i,j}$$
the rank $n(n-1)/2$ torsion-free additive semigroup with base elements $\epsilon_{i,j}$. For $$a=\sum_{1\leq j<i\leq n}a_{i,j} \epsilon_{i,j}\in\Gamma,$$ let
\begin{equation*}
E^a:=E_{2,1}^{a_{2,1}}E_{3,1}^{a_{3,1}}E_{3,2}^{a_{3,2}}E_{4,1}^{a_{4,1}}\ldots E_{n,1}^{a_{n,1}}\ldots E_{n,n-1}^{a_{n,n-1}}\in U(\frg).
\end{equation*}
The set of all monomials $E^a (a\in \Gamma)$ form the PBW basis of $U(\bar\frn)$. So each vector $v\in M(\lambda)$
can be described as
\begin{equation*}
v=\sum_{a\in \Gamma}{c_aE^av_\lambda}.
\end{equation*}

Now we recall the differential system derived in \cite{Xu}, which determines singular vectors in Verma modules over $\frsl(n,\bbC)$.

Start with the polynomial algebra
\[
\caA=\bbC[x_{i,j}\ |\ 1\leq j<i\leq n]
\]
and its basis
\[
\{x^a:=\prod_{1\leq j<i\leq n}x_{i,j}^{a_{i,j}}\ |\ a\in \Gamma\}.
\]
There exists a linear isomorphism $\tau:M(\lambda)\ra \caA$ such that
\[
\tau(E^av_\lambda)=x^a\qquad\mbox{for}\ a\in\Gamma.
\]
Thus $\caA$ has a $\frsl(n,\bbC)$-module structure given by the action
\[
A(f)=\tau(A(\tau^{-1}(f)))\qquad\mbox{for}\ A\in\frsl(n,\bbC),\ f\in\caA.
\]
Denote $\pt_{i,j}=\pt/\pt x_{i,j}$. Then
\begin{equation}\label{eqet1}
\eta_i:=E_{i+1,i}|_\caA=x_{i+1,i}+\sum_{j=1}^{i-1}x_{i+1,j}\pt_{i,j}
\end{equation}
for $i=1,2,\ldots,n-1$. Moreover, we have
\begin{equation*}
\begin{aligned}
d_i:=&E_{i,i+1}|_\caA\\
=&\left(\lambda_i-1-\sum_{i+1}^nx_{j,i}\pt_{j,i}+\sum_{j=i+2}^nx_{j,i+1}\pt_{j,i+1}\right)\pt_{i+1,i}\\
&+\sum_{j=1}^{i-1}x_{i,j}\pt_{i+1,j}-\sum_{j=i+2}^nx_{j,i+1}\pt_{j,i},
\end{aligned}
\end{equation*}
and the weight operators
\begin{equation*}
\begin{aligned}
\zeta_i:=&H_{i}|_\caA\\
=&\lambda_i-1-\sum_{j=1}^{i-1}(x_{i,j}\pt_{i,j}-x_{i+1,j}\pt_{i+1,j})\\
&+\sum_{j=i+2}^{n}(x_{j,i+1}\pt_{j,i+1}-x_{j,i}\pt_{j,i})-2x_{i+1,i}\pt_{i+1,i}
\end{aligned}
\end{equation*}
for $i=1,2,\ldots,n-1$. Here $$\lambda_i:=\langle\lambda+\rho, e_i-e_{i+1}\rangle=(\lambda+\rho)(H_i).$$

\begin{remark}
Normally, $\lambda_i$ is assigned to $\lambda(H_i)$. However, it seems that the notation could be effectively simplified if $\lambda_i$ is defined to be $(\lambda+\rho)(H_i)$ rather than $\lambda(H_i)$.
\end{remark}

\begin{prop}[\cite{Xu}, Proposition 2.1]
Given a weight vector $v\in M(\lambda)$, it is a singular vector if and only if
\[
d_i(\tau(v))=0\qquad\mbox{for}\ i=1,\ldots,n-1.
\]
\end{prop}

\begin{definition}
We can define a system of partial differential equations
\begin{equation}\label{eqd2}
d_i(f)=0
\end{equation}
for $i=1,\ldots,n-1$ and unknown function $f$ in $\{x_{i,j}\ |\ 1\leq j<i\leq n\}$. It is called the {\it system of partial differential equations for the singular vectors of $\frsl(n,\bbC)$}.
\end{definition}
Given a function $f$ of $\{x_{i,j}\ |\ 1\leq j<i\leq n\}$, we say $f$ is {\it weighted} if there exist $\mu\in\frh^*$ such that $\zeta_i(f)=\mu(H_i)f$ for all $i=1,\ldots,n-1$. For a weighted polynomial solution $f$ of $(\ref{eqd2})$, $\tau^{-1}(f)$ is a singular vector of $M(\lambda)$. In particular, since
\[
\zeta_i(1)=\tau(H_iv_\lambda)=\lambda(H_i)\tau(v_\lambda)=\lambda(H_i)1,
\]
the weight of the constant polynomial $1$ is $\lambda$.

To solve this system, we need a proper space of functions. Denote by
\[
\caA_0:=\bbC[x_{i,j}\ |\ 1\leq j< i-1\leq n-1]
\]
the polynomial algebra in $\{x_{i,j}\ |\ 1\leq j< i-1\leq n-1\}$ and
\[
\textbf{x}^z:=\prod_{i=1}^{n-1}x_{i+1,i}^{z_i}\quad\mbox{for}\ z=(z_1,z_2,\ldots,z_{n-1})\in\bbC^{n-1}.
\]
We define
\[
\caA_1:=\left\{\sum_{b\in\bbN^{n-1}}\sum_{i=1}^pf_{z^i-b}\textbf{x}^{z^i-b}\ |\ 1\leq p\in\bbN, z^i\in\bbC^{n-1}, f_{z^i-b}\in\caA_0\right\}
\]
to be the space of truncated-up formal power series in $\{x_{2,1},x_{3,2}\ldots,x_{n,n-1}\}$ over $\caA_0$. It is evident that $\caA_1$ is invariant under the action of $\{\zeta_i,d_i,\eta_i\ |\ i=1,\ldots,n-1\}$ and hence a natural $\frsl(n,\bbC)$-module.

Given $\gamma\in\bbC$, denote
\[
\langle \gamma\rangle_k=\gamma(\gamma-1)\ldots(\gamma-k+1)
\]
for any $k\in\bbN$. In particular, if $k=0$, then $\langle \gamma\rangle_0=1$. In view of Eq. \ref{eqet1}, we define differential operators
\begin{equation}\label{eqet2}
\eta_i^c=\left(x_{i+1,i}+\sum_{j=1}^{i-1}x_{i+1,j}\pt_{i,j}\right)^c
=\sum_{p=1}^\infty\frac{\langle c\rangle_p}{p!}x_{i+1,i}^{c-p}\left(\sum_{j=1}^{i-1}x_{i+1,j}\pt_{i,j}\right)^p.
\end{equation}
for $i=1,2,\ldots,n-1$ and $c\in\bbC$. An immediate consequence of the definition is
\[
\eta_i^{c_1}\eta_i^{c_2}=\eta_i^{c_1+c_2}\qquad\mbox{for}\ c_1,c_2\in\bbC
\]
since $x_{i+1,i}$ and $\sum_{j=1}^{i-1}x_{i+1,j}\pt_{i,j}$ are commutative.

\begin{lemma}[\cite{Xu}, Lemma 2.2 and Lemma 2.3]
If $i,j\in\{1,2,\ldots,n-1\}$ and $c\in \bbC$, then
\[
[d_i,\eta_j^c]=c\delta_{i,j}\eta_j^{c-1}(1-c+\zeta_j)\quad\mbox{and}\quad
[\zeta_i,\eta_j^{c}]=-cA_{i,j}\eta_j^c,
\]
where $A_{i,j}$ are the $(i,j)$ entry of the Cartan matrix of $\frsl(n,\bbC)$.
\end{lemma}

\begin{lemma}[\cite{Xu}, Lemma 2.4]
If $1\leq i<n-1$ and $c_1,c_2\in\bbC$, then
\[
\eta_i^{c_1}\eta_{i+1}^{c_1+c_2}\eta_i^{c_2}=\eta_{i+1}^{c_2}\eta_{i}^{c_1+c_2}\eta_{i+1}^{c_1}.
\]
\end{lemma}

For convenience, denote $s_i:=s_{e_i-e_{i+1}}$. Then the symmetric group is generated by $\{s_1,\ldots,s_{n-1}\}$. Now we give an action of $\{s_i\}$ on $\caA_1$. If $f\in\caA_1$ is of weight $\mu$ (that is, $\zeta_j(f)=\mu(H_j)f$ for $j=1,\ldots,n-1$), then we define
\begin{equation*}
s_i(f)=\eta_i^{(\mu+\rho)(H_i)}f=\eta_i^{\mu_i}f.
\end{equation*}
In general, if $f=\sum_{j\in\bbZ}f_j\in\caA_1$ so that $f_j$ are weighted. We define
\begin{equation}\label{eqact1}
s_i(f)=\sum_{j\in\bbZ}s_i(f_j),
\end{equation}
for $i=1,2,\ldots,n-1$.
\begin{theorem}[\cite{Xu}, Theorem 2.5 and Theorem 3.1]
By $(\ref{eqact1})$, the space $\caA_1$ is a representation of the symmetric group $S_n$. The solution space of the system $(\ref{eqd2})$ is the span of $\{\sigma(1)\ |\ \sigma\in S_n\}$, which is the set of all weighted solutions of $(\ref{eqd2})$. Moreover, if $\lambda$ is dominant integral, there are up to a scalar $n!$ singular vectors in the Verma module $M(\lambda)$.
\end{theorem}

%
%%%%%%%%%%%%%%%%%%%%%%%%%%%%%%%%%%%%%%%%%%%%%%%%%%%%%%%%%%%%%%%%%%%%
%
\section{Singular vectors in Verma modules}
%
%%%%%%%%%%%%%%%%%%%%%%%%%%%%%%%%%%%%%%%%%%%%%%%%%%%%%%%%%%%%%%%%%%%%
%
In this, we present the formula of $s_\alpha(1)$ for positive root $\alpha$ and showed directly that it is a polynomial if and only if $\langle\lambda+\rho,\alpha\rangle$ is a nonnegative integer. An explicit formula of singular vectors in Verma modules of $\frsl(n,\bbC)$ is also obtained. First, we give an example of $s_\alpha(1)$.

\begin{example}
If $\alpha=e_1-e_3$, then
\[
\begin{aligned}
s_\alpha(1)=&s_1s_2s_1(1)=\eta_1^{\lambda_2}\eta_2^{\lambda_1+\lambda_2}\eta_1^{\lambda_1}(1)\\
=&x_{2,1}^{\lambda_2}(x_{3,2}+x_{3,1}\pt_{2,1})^{\lambda_1+\lambda_2}x_{2,1}^{\lambda_1}(1)\\
=&x_{2,1}^{\lambda_2}\sum_{p=0}^\infty\frac{\langle\lambda_1+\lambda_2\rangle_{p}}{p!}
x_{3,2}^{\lambda_1+\lambda_2-p}(x_{3,1}\pt_{2,1})^px_{2,1}^{\lambda_1}\\
=&\sum_{p=0}^\infty\frac{\langle\lambda_1+\lambda_2\rangle_{p}\langle\lambda_1\rangle_p}{p!}
x_{3,2}^{\lambda_1+\lambda_2-p}x_{2,1}^{\lambda_1+\lambda_2-p}x_{3,1}^p.
\end{aligned}
\]
\end{example}

Next we state the first result.

\begin{theorem}\label{main thm1}
Suppose that $\alpha=e_k-e_l$ for $1\leq k<l\leq n$. Denote
\[
\Gamma_{k,l}:=\{a\in\Gamma\ |\ a_{i,j}=0\ \mbox{if}\ i=j+1, i>l\ \mbox{or}\ j<k\}.
\]
If $a\in \Gamma_{k,l}$, define
$r_i=\sum_{j< i+1<q}{a_{q,j}}$, $t_i=r_i+\sum_{j< i}{a_{i+1,j}}$ and $ u_{i}=\sum_{j=k}^i\lambda_j$ for $i=k,\ldots,l-1$ $(r_{l-1}=0)$. Then
\begin{equation*}
s_{\alpha} (1)=\sum_{a\in\Gamma_{k,l}}\frac{\prod_{i=k}^{l-1}\langle u_i\rangle_{r_i}\langle u_{l-1}-r_i\rangle_{t_i-r_i}x_{i+1,i}^{ u_{l-1}-t_i}}{\prod_{i>j}{a_{i,j}!}}x^a.
\end{equation*}
In particular, $s_{\alpha} (1)$ is a polynomial if and only if $ u_{l-1}\in\bbN$.
\end{theorem}

The first part of theorem can be proved in a more general setting.

\begin{lemma}\label{main thm3}
Let $f$ be a weighted element in $\caA_1$ with weight $\mu$. We use the same notation in the above theorem except $u_{i}=\sum_{j=k}^i\mu_j$ for $i=k,\ldots,l-1$ $(\mu_j=(\mu+\rho)(H_j))$. If $\eta_i(f)=x_{i+1,i}f$ for $i=k,\ldots, l-1$, then
\begin{equation}\label{main3eq}
s_{\alpha} (f)=f\sum_{a\in\Gamma_{k,l}}c_ax^a\prod_{i=k}^{l-1}x_{i+1,i}^{ u_{l-1}-t_i}.
\end{equation}
Here, the coefficients
\begin{equation*}\label{main3eq0}
c_a=\frac{\prod_{i=k}^{l-1}\langle u_i\rangle_{r_i}\langle u_{l-1}-r_i\rangle_{t_i-r_i}}
{\prod_{i>j}{a_{i,j}!}}.
\end{equation*}
\end{lemma}

\begin{proof}
Fix $k$ and use induction on $l>k$. If $\eta_i(f)=x_{i+1,i}f$, we see from the definition that $\eta_i^{c}(f)=x_{i+1,i}^{c}f$ for $c\in \bbC$. If $l=k+1$, one has $\Gamma_{k,k+1}=\{0\}$, $r_k=t_k=0$ and $ u_k=\mu_k$. Then $s_{e_k-e_{k+1}}(f)=\eta_k^{\mu_k}(f)=x_{k+1,k}^{\mu_k}f$. Hence $(\ref{main3eq})$ is true. Now suppose $l>k+1$. Then
\begin{equation}\label{main3eq1}
s_{e_k-e_l}(f)=s_{l-1}s_{e_k-e_{l-1}}s_{l-1}(f)
=\eta_{l-1}^{ u_{l-2}}s_{e_k-e_{l-1}}
(\eta_{l-1}^{\mu_{l-1}}f).
\end{equation}
Denote $f':=x_{l,l-1}^{\mu_{l-1}}f$.
The weight of $f'$ is $\mu'=\mu-\mu_{l-1}(e_{l-1}-e_l)$. It follows from $\eta_i(f)=x_{i+1,i}f$ for $i=k,\ldots,l-1$ that $\eta_{l-1}^{\mu_{l-1}}f=x_{l,l-1}^{\mu_{l-1}}f$ and
\[
\eta_i(f')=[\eta_i,x_{l,l-1}^{\mu_{l-1}}](f)+x_{l,l-1}^{\mu_{l-1}}\eta_i(f)
=x_{l,l-1}^{\mu_{l-1}}x_{i+1,i}f=x_{i+1,i}f'
\]
for $i=k,\ldots,l-2$. So the induction hypothesis for $l-1$ can be applied, showing that
\begin{equation}\label{main3eq2}
s_{e_k-e_{l-1}}
(f')=f'\sum_{a'\in\Gamma_{k,l-1}}c_{a'}x^{a'}
\prod_{i=k}^{l-2}x_{i+1,i}^{ u'_{l-2}-t'_i},
\end{equation}
with
\begin{equation}\label{main3eq3}
c_{a'}=\frac{\prod_{i=k}^{l-2}\langle u'_i\rangle_{r'_i}\langle u'_{l-2}-r'_i\rangle_{t'_i-r'_i}}
{\prod_{i>j}{a'_{i,j}!}}.
\end{equation}
In view of $(\ref{eqet2})$, one has
\begin{equation}\label{main3eq4}
\eta_{l-1}^{ u_{l-2}}
=\sum_{p_{l,j}\in\bbN}\langle u_{l-2}\rangle_{p}
x_{l,l-1}^{ u_{l-2}-p}\prod_{j=1}^{l-2}\frac{(x_{l,j}\pt_{l-1,j})^{p_{l,j}}}{p_{l,j}!},
\end{equation}
where $p=p_{l,1}+\ldots+p_{l,l-2}$.
Substituting (\ref{main3eq2}) and (\ref{main3eq4}) into (\ref{main3eq1}) , we obtain
\begin{equation}\label{main3eq5}
s_{e_k-e_l}(f)=f\sum_{a\in\Gamma_{k,l}}\bar c_{a}
x^{a}
x_{l,l-1}^{ u_{l-1}-p}x_{l-1,l-2}^{ u'_{l-2}-t'_{l-2}-p_{l,l-2}}\prod_{i=k}^{l-3}x_{i+1,i}^{ u'_{l-2}-t'_i},
\end{equation}
where \begin{equation}\label{main3eq6}
a=a'+\sum_{j=1}^{l-2}p_{l,j}\epsilon_{l,j}-\sum_{j=1}^{l-3}p_{l,j}\epsilon_{l-1,j}
\end{equation}
and
\begin{equation}\label{main3eq7}
\bar c_a=c_{a'}\frac{\langle u_{l-2}\rangle_{p}\langle u'_{l-2}-t'_{l-2}\rangle_{p_{l,l-2}}
\prod_{j=k}^{l-3}\langle a'_{l-1,j}\rangle_{p_{l,j}}}{\prod_{j=k}^{l-2}p_{l,j}!}.
\end{equation}
We need to explain why the sum in $(\ref{main3eq5})$ is taken over $a\in\Gamma_{k,l}$. It suffices to show that $\bar c_a=0$ if $a\not\in\Gamma_{k,l}$. In fact, if $a_{l-1,j}=a'_{l-1,j}-p_{l,j}<0$ for $j\in\{1,\ldots,l-3\}$, then $\langle a'_{l-1,j}\rangle_{p_{l,j}}=0$ and $\bar c_a=0$.

Now we can assume that $a\in\Gamma_{k,l}$. Since $a'\in\Gamma_{k,l-1}$, it follows from $(\ref{main3eq6})$ that
\begin{equation}\label{main3eq8}
a_{i,j}=\left\{\begin{aligned}
&a'_{i,j}-p_{l,j}\qquad\quad\mbox{if}\ i=l-1,\ k\leq j< l-2\\
&p_{l,j}\qquad\qquad\qquad\mbox{if}\ i=l,\ k\leq j< l-1\\
&a'_{i,j}\qquad\qquad\qquad\ \mbox{otherwise}.
\end{aligned}
\right.
\end{equation}
In particular, one has
\begin{equation}\label{main3eq9}
(r_i,\ t_i)=\left\{\begin{aligned}
&(r'_i,\ t'_i)\qquad\qquad\qquad\mbox{if}\ k\leq i< l-2\\
&(p,\ t'_{l-2}+p_{l,l-2})\qquad\mbox{if}\ i=l-2\\
&(0,\ p)\qquad\qquad\qquad\ \mbox{if}\ i=l-1.
\end{aligned}
\right.
\end{equation}
On the other hand, since $\mu'=\mu-\mu_{l-1}(e_{l-1}-e_l)$, we obtain
\begin{equation}\label{main3eq10}
 u'_i=\left\{\begin{aligned}
& u_i\qquad\qquad\quad\mbox{if}\ k\leq i< l-2\\
& u_{l-1}\qquad\qquad\mbox{if}\ i=l-2.
\end{aligned}
\right.
\end{equation}
With $(\ref{main3eq8})$, $(\ref{main3eq9})$ and $(\ref{main3eq10})$ in hand, we substitute $(\ref{main3eq3})$ into $(\ref{main3eq7})$ and get $c_a=\bar c_a$. Therefore $(\ref{main3eq})$ follows from $(\ref{main3eq5})$ by applying $(\ref{main3eq9})$ and $(\ref{main3eq10})$.
\end{proof}

{\bf Proof of Theorem \ref{main thm1}} The first statement follows from Lemma \ref{main thm3} and the fact that $\eta_i(1)=x_{i+1,i}$. If $s_{\alpha} (1)$ is a polynomial, then one has $ u_{l-1}-t_i\in\bbN$. Since $t_i\in\bbN$, we obtain $ u_{l-1}\in\bbN$. Conversely, if $u_{l-1}\in\bbN$ and $s_{\alpha} (1)$ is not a polynomial, then there exists $a\in\Gamma_{k,l}$ with $c_a\neq0$ and $u_{l-1}-r_i<0$ for some $i\in\{k,\ldots,l-1\}$. Choose the largest $i$ such that $u_{l-1}-r_i<0$. Since $u_{l-1}-r_{l-1}=u_{l-1}\geq0$, then $i<l-1$ and $u_{l-1}-r_{i+1}\geq0$. Note that
\[
t_{i+1}-r_i=\sum_{j< i+2<q}{a_{q,j}}+\sum_{j< i+1}{a_{i+2,j}}-\sum_{j< i+1<q}{a_{q,j}}=\sum_{q>i+2}{a_{q,i+1}}\geq0.
\]
It follows that
\[
u_{l-1}-t_{i+1}+1\leq u_{l-1}-r_{i}+1\leq0.
\]
and thus
\[
\langle{u_{l-1}-r_{i+1}}\rangle_{t_{i+1}-r_{i+1}}=0
\]
Therefore $c_a=0$ and this leads to a contradiction.

\begin{theorem}\label{main2}
Given $k,l\in \bbN$ such that $1\leq k<l\leq n$, suppose that $\langle\lambda+\rho,
e_{k}-e_{l}\rangle=m\in\bbN$. Denote
\[
\Gamma_{k,l}^{m}:=\{a\in \Gamma\ |\
[H,E^a]={m}(e_{l}-e_{k})(H)E^a\ \mbox{for all}\ H\in\frh\}.
\]
If $a\in \Gamma$, define
$r_i=\sum_{j< i+1<q}{a_{q,j}}$ and $u_{i}=\sum_{j=k}^i\lambda_j$ for $i=k,\ldots,l-1$ $(r_{l-1}=0)$. Then
\begin{equation}\label{main1}
v=\sum_{a\in\Gamma_{k,l}^m}\frac{\prod_{i=k}^{l-1}\langle u_i\rangle_{r_i}( m-r_i)!}{\prod_{i>j}{a_{i,j}!}}E^av_\lambda,
\end{equation}
is a singular vector in $M(\lambda)$ of weight $s_{e_{k}-e_{l}}\cdot\lambda$.
\end{theorem}

\begin{proof}
Rather than check $E_{i,i+1}\cdot v=0$ directly for all
$i\in\{1,\ldots,n\}$ (which is exposed thoroughly in \cite{X}), we shall provide a more natural proof using Theorem \ref{main thm1}.

For convenience we define a vector
\[
\vf(a')=a'+\sum_{i=k}^{l-1}(m-t'_i)\epsilon_{i+1,i}
\]
for each $a'\in\Gamma_{k,l}$. Denote $\alpha:=e_k-e_l$. Since $u_{l-1}=m\in\bbN$, it follows from Theorem \ref{main thm1} that $s_\alpha(1)$ is a polynomial. Therefore, we have a singular vector
\begin{equation}\label{main2eq1}
v=\tau(s_\alpha(1))=\sum_{a'\in\Gamma_{k,l},\vf(a')\in\Gamma}\frac{\prod_{i=k}^{l-1}\langle u_i\rangle_{r'_i}\langle m-r'_i\rangle_{t'_i-r'_i}}{\prod_{i>j}{a'_{i,j}!}}
E^{\vf(a')}v_\lambda.
\end{equation}
We claim that
\begin{equation}\label{main2eq2}
\vf(\Gamma_{k,l})\cap\Gamma=\Gamma_{k,l}^m.
\end{equation}
In fact, fix $a'\in\Gamma_{k,l}$, and suppose $\vf(a')\in\Gamma$. Set $a:=\vf(a')$ and $t_i:=r_i+\sum_{j< i}{a_{i+1,j}}$. Then $a_{q,j}=a'_{q,j}$ except that $a_{i+1,i}=m-t'_i$ and $a'_{i+1,i}=0$ for $i=k,\ldots,l-1$. It follows that
\begin{equation}\label{main2eq3}
r_i=r'_i,\quad t_i=t'_i
\end{equation}
for $i=k,\ldots,l-1$ and
\begin{equation}\label{main2eq4} \prod_{i>j}a_{i,j}!=\prod_{i=k}^{l-1}(m-t'_i)!\prod_{i>j}a'_{i,j}!.
\end{equation}
Now we prove $(\ref{main2eq2})$. Since the weight of $s_\alpha(1)$ is $s_\alpha\cdot\lambda$, so are the monomials $x^{\vf(a')}$ and $\tau(x^{\vf(a')})=E^{\vf(a')}v_\lambda$. Therefore $a=\vf(a')\in\Gamma_{k,l}^m$.
On the other hand, if $a\in\Gamma_{k,l}^m$, then
\[
\sum_{q>j}a_{q,j}[(e_q-e_{q-1})+\ldots+(e_{j+1}-e_j)]=m[(e_l-e_{l-1})+\ldots+(e_{k+1}-e_k)].
\]
Taking the coefficients of $e_{i+1}-e_i$ of both sides, we have
\[
\sum_{j<i+1\leq q}a_{q,j}=m
\]
and thus $a_{i+1,i}=m-t_i$ for $i=k,\ldots,l-1$. Let $a'=a-\sum_{i=k}^{l-1}(m-t_i)\epsilon_{i+1,i}$. We still have $(\ref{main2eq3})$. Thus $a'\in\Gamma_{k,l}$ and $\vf(a')=a$.

With $(\ref{main2eq2})$ in hand, substitute $(\ref{main2eq3})$ and $(\ref{main2eq4})$ into $(\ref{main2eq1})$. We obtain
\[
\begin{aligned}
v=&\sum_{a\in\Gamma_{k,l}^m}\frac{\prod_{i=k}^{l-1}\langle u_i\rangle_{r_i}\langle m-r_i\rangle_{t_i-r_i}(m-t_i)!}{\prod_{i>j}{a_{i,j}!}}
E^{a}v_\lambda\\
=&\sum_{a\in\Gamma_{k,l}^m}\frac{\prod_{i=k}^{l-1}\langle u_i\rangle_{r_i}(m-r_i)!}{\prod_{i>j}{a_{i,j}!}}
E^{a}v_\lambda.
\end{aligned}
\]

\end{proof}

\begin{example}\label{example coe1}
Let $\frg=\frsl(4, \mathbb{C})$. Fix $\lambda\in\frh^*$ such that $\langle\lambda+\rho,
e_1-e_4\rangle=1$. Then $\lambda_1+\lambda_2+\lambda_3=1$. So
\[
\Gamma_{1,4}^1=\{\epsilon_{2,1}+\epsilon_{3,2}+\epsilon_{4,3},\epsilon_{3,1}+\epsilon_{4,3},
\epsilon_{2,1}+\epsilon_{4,2},\epsilon_{4,1}\}.
\]

It follows from Theorem \ref{main2} that the singular vector of weight $s_{e_1-e_4}\cdot\lambda$ in $M(\lambda)$
is (up to a scalar)
\begin{equation*}
\begin{aligned}
v=&E_{2,1}E_{3,2}E_{4,3}v_\lambda +\lambda_1 E_{3,1}E_{4,3}v_\lambda\\
&+(\lambda_1+\lambda_2) E_{2,1}E_{4,2}v_\lambda +\lambda_1(\lambda_1+\lambda_2)
E_{4,1}v_\lambda.
\end{aligned}
\end{equation*}
\end{example}

\begin{remark}
If $k=1$ and $l=n$, the above result is essentially Theorem 5.1 in \cite{MFF}, although the Eq. \ref{main1} of singular vector is clearer.
\end{remark}

\begin{cor}\label{main cor}
Given $\lambda,\mu\in\frh^*$. Suppose that $\mu$ is strongly linked to $\lambda$ by $\alpha_k,\ldots,\alpha_1\in \Phi^+$. If $$\lambda-s_{\alpha_k}\ldots s_{\alpha_1}\cdot\lambda=\sum_{i=1}^{n-1}a_i(e_{i+1}-e_{i}),$$ then there exists a singular vector $v$ of weight $s_{\alpha_k}\ldots s_{\alpha_1}\cdot\lambda$ in $M(\lambda)$. In particular,
\begin{equation*}
v=E_{2,1}^{a_{1}}\ldots E_{n,n-1}^{a_{n-1}}v_\lambda+\sum_{\scriptstyle a\in\Gamma,\atop\scriptstyle |a|<a_{1}+\ldots+a_{n-1}}c_aE^av_\lambda,
\end{equation*}
where $|a|:=\sum_{i>j}a_{i,j}$ is the degree of $E^a$.
\end{cor}

\begin{proof}
This can be proved by induction on $k$, using Theorem \ref{main2}.
\end{proof}

\begin{remark}
By Theorem \ref{BGG theorem},
the above corollary substantially gives all the
singular vectors in $M(\lambda)$.
\end{remark}
%

%\begin{lemma}
%For $\frg=\frgl(m|n)$, the constant $C=\|\rho\|^2-\|\rho_0\|^2$.
%\end{lemma}

%\begin{proof}
%In view of Proposition \ref{D2}, $C$ is $1/8$ of the trace of Casimir element %$\Omega_{\frg_0}$ on $\frg_1$.
%\end{proof}

\end{document}